\documentclass[a4paper,11pt]{article}
\usepackage{hyperref}
\hypersetup{
    colorlinks=true,
    linktoc=all,
    linkcolor=red,     
    citecolor=blue,     
}
\usepackage{graphicx}
\usepackage{amssymb}
\usepackage{bm}
\usepackage{multicol}
\usepackage{indentfirst}
\usepackage{amsfonts}
\usepackage{amsmath}

\usepackage{setspace}
\newenvironment{pf}[1][Proof]{\noindent {\bf #1.}}{\hfill\rule{2mm}{3mm}\par\medskip}

\renewcommand{\arraystretch}{0.9}

\newtheorem{theorem}{Theorem}[section]
\newtheorem{lemma}{Lemma}[section]

\newtheorem{corollary}{Corollary}[section]

\newtheorem{definition}{Definition}[section]

\newtheorem{remark}{Remark}[section]

\newcommand{\ignore}[1]{}
\newcommand{\srg}{\operatorname{srg}}

\begin{document}
\begin{spacing}{1.0}

\title{Book Ramsey numbers via algebraic constructions\footnote{Center for Discrete Mathematics, Fuzhou University,
Fuzhou, 350108, P.~R.~China. Email: {\tt 1415088965@qq.com, linqizhong@fzu.edu.cn}. Supported in part by National Key R\&D Program of China (Grant No. 2023YFA1010202) and NSFC (No.\ 12571361).}}
\date{}

\author{
Lulu Dai\;\; and \;\; Qizhong Lin
}

\maketitle
\begin{abstract}
Let \(B_n\) denote the book graph consisting of \(n\) triangles sharing a common edge. Rousseau and Sheehan (1978) proved, using Paley graphs, that \(R(B_n,B_n)=4n+2\) whenever \(4n+1\) is a prime power. Since then, only a few additional exact values of \(R(B_n,B_n)\) have been obtained.

In this paper, we obtain \(R(B_n,B_n)=4n+1\) for infinitely many \(n\) by constructing graphs from pseudo-cyclic strongly regular graphs (PC-graphs). To the best of our knowledge, this is the first new infinite family of exact diagonal book Ramsey numbers since 1978.
We also prove that \(R(B_{n-2},B_n)\le 4n-3\) for every \(n\ge 3\) with \(n\ne 6\), removing the original condition \(n\equiv 2\pmod 3\) imposed by Rousseau and Sheehan. Moreover, equality \(R(B_{n-2},B_n)=4n-3\) holds if there exists a symmetric Hadamard matrix of order \(2n-2\) with all diagonal entries equal to \(1\). In particular, this equality holds for every \(n=2^{2\ell-1}+1\) with \(\ell\ge 1\).
\end{abstract}

\section{Introduction}

For a positive integer $n$, the book graph $B_n$ is defined as $K_2 + \overline{K}_n$, i.e., $n$ triangles sharing a common edge. The Ramsey number $R(B_m, B_n)$ is the smallest integer $N$ such that every red/blue coloring of the edges of $K_N$ contains either a red copy of $B_m$ or a blue copy of $B_n$.

The study of Ramsey numbers for books dates back to Erd\H{o}s, Faudree, Rousseau and Schelp \cite{efrs} and to Rousseau and Sheehan \cite{RS78} in 1978. In particular, they established the bounds
\[
(4+o(1))n \le R(B_n, B_n) \le 4n+2.
\]
Moreover, Rousseau and Sheehan \cite{RS78} proved that whenever $4n+1$ is a prime power,
\[
R(B_n, B_n) = 4n + 2,
\]
using the Paley graph construction. They further showed that if $4n+1$ cannot be expressed as a sum of two squares, then $R(B_n, B_n) \le 4n+1$, with the first concrete example $R(B_5,B_5)=21$ obtained via an explicit but non-strongly-regular coloring.

Since that work, progress on exact diagonal book Ramsey numbers has
been remarkably limited. Beyond these classical results, only a small
number of additional cases have been settled, and no further infinite
family of exact values seems to have been known. Thus, finding new
infinite families remains a challenging problem.

For the off-diagonal case, the results are considerably richer. Rousseau and Sheehan \cite{RS78} obtained the first exact value $R(B_1, B_n) = 2n + 3$ for $n \ge 2$. This was later extended to larger parameters: Faudree, Rousseau and Sheehan \cite{FRS} proved that $R(B_m, B_n) = 2n + 3$ for $n\ge C m^4$, where $C>0$ is an absolute constant. Subsequently, Nikiforov and Rousseau \cite{NR1} showed that the equality already holds when $n$ is linear in $m$. Finally, Nikiforov and Rousseau \cite{NR05} used a stability result on books by Bollob\'{a}s and Nikiforov \cite{BN} to prove that for any fixed $0<\alpha < 1/6$ and sufficiently large $n$,
\[
R(B_{\lceil \alpha n\rceil}, B_n) = 2n + 3,
\]
and that the constant $1/6$ is asymptotically best possible. On the other hand, Conlon, Fox and Wigderson \cite{CFW23} proved that for $\alpha$ sufficiently close to $1$, the random lower bound is asymptotically tight:
\[
R(B_{\alpha n}, B_n) = (\sqrt{\alpha} + 1)^2 n + o(n).
\]
Fan, Lin and Yan \cite{FLY} showed that the random lower bound is asymptotically tight for any fixed $1/4 \le \alpha \le 1$ and sufficiently large $n$; the special case $\alpha = 1/4$ was also established in \cite{CL}. 
In forthcoming work, we will determine the asymptotic value of $R(B_{\alpha n},B_n)$ for the remaining range of $\alpha$.

Despite this progress on asymptotic estimates, exact values in the nearly diagonal range remain scarce. Rousseau and Sheehan \cite{RS78} proved that $R(B_{k^2-2}, B_{k^2+1}) = 4k^2$ for all integers $k\ge 2$ of the form $k = 3^{m}2^{m+r-1}$, where $m,r$ are non-negative integers. They also proved the small cases $R(B_2, B_5) = 16$, $R(B_3, B_5) = 17$, and $R(B_4, B_6) = 22$. In recent work, Wesley~\cite{Wes} improved lower bounds using block-circulant graphs and SAT solvers, proving, in particular, that $R(B_{n-1},B_n)=4n-1$ for infinitely many $n$. These computational and algebraic results motivate further study of diagonal and nearly diagonal cases.

Our first result provides a lower bound for $R(B_n, B_n)$ using pseudo-cyclic strongly regular graphs (PC-graphs; see Definition \ref{def:pc-gr}).
\begin{theorem}\label{thm:Bn-PC}
Suppose there exists a PC-graph of order $2n-1$ or of order $n$. Then
\[
R(B_n,B_n)\ge 4n+1.
\]
\end{theorem}

Our approach broadens the algebraic input from Paley graphs to arbitrary
pseudo-cyclic strongly regular graphs (PC-graphs), thereby allowing graphs
of non-prime-power order to be used. In particular, Mathon's product
construction produces PC-graphs of order $pq^2$, where
$q\equiv3\pmod4$ is a prime power and $p=q+2$ is the order of a
PC-graph.
Theorem~\ref{thm:Bn-PC} applies whenever either $2n-1$ or $n$ is the
order of such a graph. Combining the resulting lower bound with the
classical upper bound of Rousseau and Sheehan yields the following
corollary.

\begin{corollary}\label{cor:paley-mathon}
Suppose that $4n+1$ is not a sum of two integer squares. Then 
$
R(B_n,B_n)=4n+1
$ if one of the following holds:

\medskip
(i) either $2n-1$ or $n$ is a prime power congruent to $1\pmod 4$;

\medskip
(ii) either $2n-1=pq^2$ or $n=pq^2$, where
$q\equiv3\pmod4$ is a prime power and $p=q+2$ is the order of a
PC-graph.
\end{corollary}

In particular, condition~(i) already yields a simple explicit infinite
family. For every integer $r\ge1$, let $n=(9^r+1)/2$. Then
$2n-1=9^r$ is a prime power congruent to $1\pmod4$, while
\[
4n+1=2\cdot9^r+3=3\bigl(2\cdot3^{2r-1}+1\bigr)
\]
is divisible by $3$ but not by $9$. Since the prime $3\equiv3\pmod4$ occurs in its factorization with odd multiplicity, Fermat's two-square theorem shows that $4n+1$ is not a sum of two squares.
Hence 
\begin{equation}\label{eq:explicit-diagonal-family}
R\!\left(B_{(9^r+1)/2},B_{(9^r+1)/2}\right)=2\cdot9^r+3
\qquad (r\ge1).
\end{equation}
Condition~(ii) extends the available inputs to PC-graphs obtained from
Mathon's construction, including graphs of non-prime-power order.

Our next result concerns the nearly diagonal case. Rousseau and Sheehan~\cite{RS78} proved that $R(B_{n-2},B_n)\le4n-3$ when $n\equiv2\pmod3$. We extend this bound to every integer $n\ge3$ other than $6$. The exception is necessary, as shown in Remark~\ref{rem:n6}.
Our proof shows that equality in the triangle-counting inequalities
forces any counterexample to be a strongly regular graph with parameters
\[
\srg(4n-3,2n-2,n-3,n).
\]
Integrality of the eigenvalue multiplicities then excludes every $n\ne6$.

\begin{theorem}\label{thm:near-exact}
For every integer $n\ge 3$ with $n\ne 6$,
\[
R(B_{n-2},B_n)\le 4n-3.
\]
\end{theorem}

For the special case $n=10$, Kalfus and Lidick\'{y}~\cite{KL}
proved $R(B_8,B_{10})=37$ using the AI-assisted tool AutoMath. Theorem~\ref{thm:near-exact} establishes the upper
bound uniformly for all $n\ge3$ other than $6$.
Combining the upper bound in Theorem \ref{thm:near-exact} with the lower bounds recorded by Wesley~\cite[Table~2]{Wes}
also gives
\[
R(B_{n-2},B_n)=4n-3
\quad\text{for }n\in\{7,9,10,12,13,15,16\}.
\]

We next give a lower-bound construction based on symmetric Hadamard
matrices.

\begin{lemma}\label{cor:near-hadamard-exact}
Let $n\ge 3$. If there exists a symmetric Hadamard
matrix of order $2n-2$ whose diagonal entries are all equal to $1$,
then $R(B_{n-2},B_n)\ge 4n-3.$
\end{lemma}

Together, Theorem~\ref{thm:near-exact} and
Lemma~\ref{cor:near-hadamard-exact} determine the exact value whenever
such a Hadamard matrix exists. In particular, for
$n=2^{2\ell-1}+1$ with $\ell\ge 1$, a symmetric Hadamard matrix of
order
\[
2n-2=2^{2\ell}
\]
with all diagonal entries equal to $1$ can be obtained by Kronecker
products. This yields the following corollary.

\begin{corollary}\label{cor:hadamard-exact}
For every integer $\ell\ge 1$, if $n=2^{2\ell-1}+1$, then
$
R(B_{n-2},B_n)=4n-3.
$
\end{corollary}

The paper is organized as follows. Section~\ref{sec:prelim} collects the necessary definitions and known results on strongly regular graphs and conference matrices. Section~\ref{sec:proof-diagonal} proves Theorem~\ref{thm:Bn-PC} and then derives Corollary~\ref{cor:paley-mathon}. Section~\ref{sec:proof-near-exact} is devoted to the nearly diagonal case: Subsection~\ref{sec:proof-upper} proves the improved upper bound (Theorem~\ref{thm:near-exact}), and Subsection~\ref{sec:proof-hadamard} gives the Hadamard lower bound construction, proving Lemma~\ref{cor:near-hadamard-exact} from which Corollary~\ref{cor:hadamard-exact} follows immediately.

\section{Preliminaries}\label{sec:prelim}

We first recall the definition of a symmetric conference matrix, which plays a central role in the constructions of pseudo-cyclic strongly regular graphs.

\begin{definition}\label{def:scm}
A \emph{symmetric conference matrix} of order $N$ is a matrix $C=(c_{ij})$ satisfying
\[
C=C^{\mathsf T},\quad
c_{ii}=0,\quad
c_{ij}\in\{\pm1\}\;(i\ne j),\quad
C^2=CC^{\mathsf T}=(N-1)I.
\]
\end{definition}

It is well known that a necessary condition for the existence of a symmetric conference matrix is $N\equiv 2\pmod 4$; see e.g.~\cite{Belevitch1968,Seidel1973}.  Such matrices are intimately related to regular two-graphs, equiangular lines, and strongly regular graphs.  For a comprehensive survey we refer to \cite{Wallis1972}.

As we shall see in Lemma~\ref{lem:pc-conference}, symmetric conference
matrices of order $N$ are equivalent to PC-graphs of order $N-1$.
This equivalence makes PC-graphs a natural combinatorial tool for
constructing Ramsey lower bounds. We first recall the general definition
of a strongly regular graph.

\begin{definition}\label{def:srg}
A strongly regular graph $\srg(\nu,k,\lambda,\mu)$ is a graph on $\nu$
vertices that is $k$-regular, in which every pair of adjacent vertices
has exactly $\lambda$ common neighbors and every pair of non-adjacent
vertices has exactly $\mu$ common neighbors. Its complement is also
strongly regular, with parameters
$\srg(\nu,\nu-k-1,\nu-2k+\mu-2,\nu-2k+\lambda).$
\end{definition}

We now introduce the particular class of strongly regular graphs used
in our constructions.

\begin{definition}\label{def:pc-gr}
A \emph{pseudo-cyclic strongly regular graph} (PC-graph) of order $\nu$
is a strongly regular graph with parameters
$
\srg\!\left(
\nu,\frac{\nu-1}{2},\frac{\nu-5}{4},\frac{\nu-1}{4}
\right).
$
\end{definition}

The complement of a PC-graph is again a PC-graph with the same parameters, and PC-graphs satisfy $k=2\mu$. Their two nontrivial eigenvalues are $(-1\pm\sqrt{\nu})/2$.
In particular, Paley graphs of prime-power order $\nu\equiv1\pmod4$ are PC-graphs; this is the special case used by
Rousseau and Sheehan~\cite{RS78}.

We shall also need the standard spectral properties of strongly regular graphs; see e.g.~\cite[Theorem~9.1.2]{BrouwerHaemers2011}.

\begin{lemma}[\cite{BrouwerHaemers2011}]\label{lem:srg-spectral}
Let $G$ be a strongly regular graph $\srg(\nu,k,\lambda,\mu)$ with adjacency matrix $A$. Then
\[
A^2=(k-\mu)I+(\lambda-\mu)A+\mu J,
\]
where $I$ is the identity matrix and $J$ the all-one matrix. Moreover, every eigenvalue $x\ne k$ of $A$ satisfies the quadratic equation
$
x^2-(\lambda-\mu)x-(k-\mu)=0.
$
\end{lemma}

The following equivalence (see Seidel~\cite{Seidel1973}; Mathon~\cite{Mathon1978} gives an explicit formulation) links PC-graphs with symmetric conference matrices; it is fundamental for our lower bound constructions. We include a short proof here for completeness.

\begin{lemma}[\cite{Seidel1973,Mathon1978}]\label{lem:pc-conference}
There exists a PC-graph of order $N-1$ if and only if there exists a symmetric conference matrix of order $N$.
\end{lemma}
\begin{pf} 
~First, suppose there exists a PC-graph $G$ on $\nu=N-1$ vertices with adjacency matrix $A$, and let $S=J-I-2A$ be its Seidel matrix. By Lemma \ref{lem:srg-spectral}, for a PC-graph one has $S^2 = \nu I - J$ and $S\mathbf{1}=0$.  Construct a matrix $C$ of order $\nu+1=N$ by adding a new row and column to $S$, putting all new off-diagonal entries equal to $+1$ and the new diagonal entry $0$.  Then $C$ is symmetric, has zeros on the diagonal and $\pm1$ elsewhere, and a direct computation shows $C^2 = \nu I = (N-1)I$.  Hence $C$ is a symmetric conference matrix of order $N$.

Conversely, let $C$ be a symmetric conference matrix of order $N$.  By simultaneously multiplying suitable rows and columns by $-1$, we may assume that all off-diagonal entries in the last row and column are $+1$. Without loss of generality, we may assume that $C$ itself has this form. Hence we can write $C$ in block form as
$
C = \begin{pmatrix} S & \mathbf{1} \\ \mathbf{1}^{\mathsf T} & 0 \end{pmatrix},
$
where $\mathbf{1}$ is the all-one vector of length $N-1$.  Then $S$ is a symmetric matrix with zero diagonal and $\pm1$ off-diagonal; from $C^2=(N-1)I$ one obtains $S^2 = (N-1)I - J$ and $S\mathbf{1}=0$.  Define $A = (J-I-S)/2$.  One verifies that $A$ is a $(0,1)$-matrix and satisfies
$
A^2 = \frac{N-2}{4}(J+I) - A,\; AJ = \frac{N-2}{2}J,
$
which means that $A$ is the adjacency matrix of a PC-graph of order $N-1$ (a strongly regular graph with parameters $(N-1,\frac{N-2}{2},\frac{N-6}{4},\frac{N-2}{4})$). 
\end{pf}

For further constructions of PC-graphs, and hence of symmetric
conference matrices, we use the following theorem of
Mathon~\cite[Theorem~4.1]{Mathon1978}, which extends Turyn's product
construction~\cite{Turyn1971}.

\begin{lemma}[{\cite[Theorem~4.1]{Mathon1978}}]\label{lem:mathon-one-step}
Let $q\equiv3\pmod4$ be a prime power, and suppose that there exists
a PC-graph of order $p=q+2$. Then there exists a PC-graph of order
$pq^2=(q+2)q^2$.
\end{lemma}

\section{Diagonal lower bounds via PC-graphs}\label{sec:proof-diagonal}

In this section we prove the two diagonal lower-bound constructions and then derive the exact-value corollary stated in the introduction. 

\medskip
\begin{pf}[Proof of Theorem~\ref{thm:Bn-PC}]
We prove the two cases separately.

\medskip
\noindent\textbf{Case 1: There exists a PC-graph of order $2n-1$.}

\medskip
By Lemma~\ref{lem:pc-conference}, there exists a symmetric
conference matrix $C$ of order $N=2n$. We use $C$ to construct
a graph on $4n$ vertices such that neither the graph nor its
complement contains a copy of $B_n$.

Let $M=C+I$; write $M_{ij}$ for its entries. Then $M_{ii}=1$ and $M_{ij}=C_{ij}\in\{+1,-1\}$ for $i\neq j$, so $M$ is a symmetric $\{\pm1\}$-matrix with $1$'s on the diagonal. Since $C^2=(N-1)I$, we obtain $$M^2=(C+I)^2=C^2+2C+I=N I+2C.$$ Thus, for $i\ne j$, letting $M_i$ denote the $i$-th row of $M$, we have
\begin{equation}\label{eq:conference-inner-product}
\langle M_i,M_j\rangle=(M^2)_{ij}=2C_{ij}=2M_{ij}.
\end{equation}

We now construct a graph $G$ on $V(G)=[N]\times\{+1,-1\}$; hence  $|V(G)|=2N=4n$.
For two distinct vertices $(i,\varepsilon)$ and $(j,\delta)$ with $\varepsilon,\delta\in\{+1,-1\}$,
\[
(i,\varepsilon)(j,\delta)\in E(G) \iff i\neq j \;\text{ and }\; \varepsilon\delta = M_{ij}.
\]
In particular, for each $i$, the vertices $(i,+1)$ and $(i,-1)$ are nonadjacent in $G$, and hence adjacent in $\overline{G}$. The above construction is the Taylor double of the underlying graph; see \cite[Section~1.2.7]{BrouwerVanMaldeghem2022}.

We first count common neighbors of an edge of $G$. Let $(i,\varepsilon)(j,\delta)\in E(G)$. Then $i\ne j$ and $\varepsilon\delta=M_{ij}$. A vertex $(k,\eta)$ with $k\ne i,j$ is adjacent in $G$ to both endpoints if and only if
\[
\varepsilon\eta=M_{ik},
\quad\text{and}\quad
\delta\eta=M_{jk}.
\]
Multiplying these two equations gives the necessary condition
\[
M_{ik}M_{jk}=\varepsilon\delta=M_{ij}.
\]

Conversely, if $M_{ik}M_{jk}=M_{ij}$, then the unique choice
$\eta=\varepsilon M_{ik}$ satisfies both adjacency conditions.
Therefore the number of common neighbors in $G$ of the edge $(i,\varepsilon)(j,\delta)$ is
\[
\left|\{k\in[N]\setminus\{i,j\}:M_{ik}M_{jk}=M_{ij}\}\right|.
\]

By \eqref{eq:conference-inner-product}, we have $$\sum_{k=1}^N M_{ik}M_{jk}=\langle M_i,M_j\rangle=2M_{ij}.$$ The two terms with $k=i$ and $k=j$ are both equal to $M_{ij}$, because $M_{ii}=M_{jj}=1$. Hence
\begin{equation}\label{eq:sum-k-notequal-ij}
\sum_{k\ne i,j}M_{ik}M_{jk}=0.
\end{equation}
Since each summand in \eqref{eq:sum-k-notequal-ij} is $\pm1$,
exactly half of the $N-2$ products equal $M_{ij}$. Thus the endpoints
of every edge of $G$ have exactly $(N-2)/2=n-1$ common neighbors,
and $G$ contains no copy of $B_n$.

It remains to check the complement. First consider the edge $(i,+1)(i,-1)\in E(\overline G)$. For $k\ne i$, the vertex $(k,\eta)$ is adjacent in $\overline G$ to $(i,+1)$ if and only if $\eta=-M_{ik}$, while it is adjacent in $\overline G$ to $(i,-1)$ if and only if $\eta=M_{ik}$, but these two conditions cannot hold simultaneously. Hence this edge of $\overline G$ lies in no triangle.

Now let $(i,\varepsilon)(j,\delta)\in E(\overline G)$, $i\ne j$.
Then $\varepsilon\delta=-M_{ij}$. As above, a vertex $(k,\eta)$ with $k\ne i,j$ is a common neighbor of this edge in $\overline G$ if and only if $\varepsilon\eta=-M_{ik}$, and $\delta\eta=-M_{jk}$.
Therefore the number of common neighbors in $\overline G$ of the edge $(i,\varepsilon)(j,\delta)$ is
\[
\left|\{k\in[N]\setminus\{i,j\}:M_{ik}M_{jk}=-M_{ij}\}\right|.
\]

By \eqref{eq:sum-k-notequal-ij} and the same argument as above, exactly half of the $N-2$ indices $k\ne i,j$ satisfy this condition. Therefore every edge of $\overline G$ with $i\ne j$ has exactly $\frac{N-2}{2}=n-1$ common neighbors in $\overline G$. Hence $\overline G$ contains no copy of $B_n$.

Thus we have constructed a graph $G$ on $4n$ vertices such
that neither $G$ nor $\overline G$ contains a copy of $B_n$.
Therefore $R(B_n,B_n)\ge4n+1$.

\medskip
\noindent\textbf{Case 2: There exists a PC-graph of order $n$.}

\medskip

Let $A$ be the adjacency matrix of a PC-graph of order $n$, and set
\[
B=J-I-A,\qquad D=I+A.
\]
Thus $B$ is the adjacency matrix of the complement. Writing
\[
k=\frac{n-1}{2},\qquad \lambda=\frac{n-5}{4},\qquad \mu=\frac{n-1}{4},
\]
Lemma~\ref{lem:srg-spectral}, together with $J=I+A+B$, gives
\[
A^2=(k-\mu)I+(\lambda-\mu)A+\mu J
   =kI+\lambda A+\mu B.
\]
Moreover, since the underlying PC-graph is $k$-regular, we have $AJ=JA=kJ$. Since $B=J-I-A$, the matrices $A$ and $B$ commute, and hence
\[
AB=BA=A(J-I-A)=kJ-A-A^2=(k-1-\lambda)A+(k-\mu)B=\frac{n-1}{4}(A+B).
\]
The complement is also $k$-regular, so $BJ=JB=kJ$. Consequently,
\begin{align*}
B^2=B(J-I-A)=kJ-B-BA&=\frac{n-1}{2}(I+A+B)-B-\frac{n-1}{4}(A+B)\\
&=\frac{n-1}{2}I+\frac{n-1}{4}A+\frac{n-5}{4}B.
\end{align*}
Thus
\begin{equation}\label{eq:pc-order-n-algebra}
\begin{aligned}
A^2&=\frac{n-1}{2}I+\frac{n-5}{4}A+\frac{n-1}{4}B,\\
B^2&=\frac{n-1}{2}I+\frac{n-1}{4}A+\frac{n-5}{4}B,\\
AB&=BA=\frac{n-1}{4}(A+B).
\end{aligned}
\end{equation}

Define a graph $G$ on four vertex classes $V_1,V_2,V_3,V_4$, each of size $n$, by the following $4n\times4n$ adjacency matrix:
\begin{equation}\label{eq:four-block-adjacency}
\mathcal M=
\begin{pmatrix}
A&B&A&D\\
B&A&D&A\\
A&D&B&D\\
D&A&D&B
\end{pmatrix}.
\end{equation}
Since $A$ and $B$ are symmetric $(0,1)$-matrices with zero diagonal and $D=I+A$ occurs only in off-diagonal blocks, $\mathcal M$ is the adjacency matrix of a simple graph on $4n$ vertices.

Using \eqref{eq:pc-order-n-algebra}, together with $D=I+A$, direct block multiplication gives
\begin{equation}\label{eq:four-block-square}
\mathcal M^2=
\begin{pmatrix}
X&Y&Z&W\\
Y&X&W&Z\\
Z&W&U&V\\
W&Z&V&U
\end{pmatrix},
\end{equation}
where the individual blocks are obtained as follows:
\begin{align*}
X&=2A^2+B^2+D^2
   =3A^2+B^2+2A+I\\
 &= (2n-1)I+(n-2)A+(n-2)B,\\
Y&=AB+BA+AD+DA
   =2AB+2A+2A^2\\
 &= (n-1)I+(n-1)A+(n-1)B=(n-1)J,\\
Z&=A^2+BD+AB+D^2
   =2A^2+2AB+B+2A+I\\
 &= nI+(n-1)A+nB,
 \end{align*} 
 \begin{align*}
W&=AD+BA+AD+DB
   =2A^2+2AB+2A+B\\
 &= (n-1)I+(n-1)A+nB,\\
U&=A^2+B^2+2D^2
   =3A^2+B^2+4A+2I\\
 &=2nI+nA+(n-2)B,\\
V&=AD+DA+BD+DB
   =2A^2+2AB+2A+2B\\
 &= (n-1)I+(n-1)A+(n+1)B.
\end{align*}

For distinct vertices $u,v$, the entry $(\mathcal M^2)_{uv}$ equals
the number of their common neighbors in $G$.
For $1\le i,j\le 4$, the $(i,j)$-block of $\mathcal M$ determines
the adjacencies between $V_i$ and $V_j$, while the corresponding
$(i,j)$-block of $\mathcal M^2$ gives the common-neighbor counts
for such pairs of vertices. 

For a $(0,1)$-matrix $C$, let $\operatorname{supp}(C)=\{(r,s):C_{rs}=1\}$ be its support. Since $I+A+B=J$, the supports of $I$, $A$, and $B$
are pairwise disjoint and partition the set of all matrix positions.
Consequently, if the corresponding block of $\mathcal M^2$ has the form $\alpha I+\beta A+\gamma B,$ then its value at a position in $\operatorname{supp}(I)$,
$\operatorname{supp}(A)$, or $\operatorname{supp}(B)$ is respectively
$\alpha$, $\beta$, or $\gamma$. Thus the common-neighbor counts can be
read directly from the expressions for $X,Y,Z,W,U$, and $V$ above.
By symmetry, it suffices to consider the block positions with $i\le j$:

\[
\begin{array}{c|c|c|c}
\text{block position(s)}
&
\text{red-edge support}
&
(\mathcal M^2)_{ij}
&
|N_G(u)\cap N_G(v)|
\\
\hline
(1,1),(2,2) & A   & X & n-2\\
(3,3),(4,4) & B   & U & n-2\\
(1,2)       & B   & Y & n-1\\
(1,3),(2,4) & A   & Z & n-1\\
(1,4),(2,3) & I+A & W & n-1\\
(3,4)       & I+A & V & n-1
\end{array}
\]
Hence every edge of $G$ has at most $n-1$ common neighbors, and therefore $G$ contains no $B_n$.

We next consider the blue edges, that is, the nonedges of $G$.
Since $A$ and $B$ are $k$-regular with
$k=(n-1)/2$, while $D=I+A$ has row sum $k+1$, the degrees in $G$ are
\[
d_G(u)=
\begin{cases}
3k+(k+1)=2n-1, & u\in V_1\cup V_2,\\
2k+2(k+1)=2n, & u\in V_3\cup V_4.
\end{cases}
\]
For a blue edge $uv$, i.e., $uv\notin E(G)$, we therefore have
\[
|N_{\overline G}(u)\cap N_{\overline G}(v)|
=
4n-2-d_G(u)-d_G(v)
+|N_G(u)\cap N_G(v)|.
\]
As above, the value of $|N_G(u)\cap N_G(v)|$ is obtained from the
corresponding block of $\mathcal M^2$ on the support of the blue edge.
Thus:
\[
\begin{array}{c|c|c|c|c}
\text{block position(s)}
&
\text{blue-edge support}
&
(\mathcal M^2)_{ij}
&
|N_G(u)\cap N_G(v)|
&
|N_{\overline G}(u)\cap N_{\overline G}(v)|
\\
\hline
(1,1),(2,2) & B   & X & n-2 & n-2\\
(3,3),(4,4) & A   & U & n   & n-2\\
(1,2)       & I+A & Y & n-1 & n-1\\
(1,3),(2,4) & I+B & Z & n   & n-1\\
(1,4),(2,3) & B   & W & n   & n-1\\
(3,4)       & B   & V & n+1 & n-1
\end{array}
\]
Thus every edge of $\overline G$ also has at most $n-1$ common neighbors, so $\overline G$ contains no $B_n$. We have therefore constructed a red--blue coloring of $K_{4n}$ with no monochromatic $B_n$, and hence
\[
R(B_n,B_n)\ge 4n+1.
\]
\end{pf}

We next derive the exact-value corollary. Recall that Rousseau and Sheehan \cite[Corollary~2.2]{RS78} proved the following upper bound: if $4n+1$ is not a sum of two integer squares, then $ R(B_n,B_n)\le 4n+1$.

\medskip
\begin{pf}[Proof of Corollary~\ref{cor:paley-mathon}]
Assume first that condition~(i) holds. If either $2n-1$ or $n$ is a
prime power congruent to $1\pmod 4$, then the Paley graph of the
corresponding order is a PC-graph (see, e.g.,
Mathon~\cite[Theorem~1.2]{Mathon1978}). Hence
Theorem~\ref{thm:Bn-PC} gives $R(B_n,B_n)\ge 4n+1.$

Assume next that condition~(ii) holds. Then $q\equiv3\pmod4$ is a
prime power and there exists a PC-graph of order $p=q+2$.
Lemma~\ref{lem:mathon-one-step} therefore gives a PC-graph of order
$pq^2$. Since either $2n-1=pq^2$ or $n=pq^2$,
Theorem~\ref{thm:Bn-PC} again gives $R(B_n,B_n)\ge4n+1$.

In either case, since $4n+1$ is not a sum of two integer squares, the
upper bound of Rousseau and Sheehan yields $R(B_n,B_n)\le4n+1$. Therefore $R(B_n,B_n)=4n+1$.
\end{pf}

\subsection{Exact values from Paley and Mathon constructions}

The following table lists examples obtained from Paley graphs via condition~(i) of Corollary~\ref{cor:paley-mathon}.

\begin{center}
\small
\setlength{\tabcolsep}{5.5pt}
\renewcommand{\arraystretch}{1.15}
\begin{tabular}{c|ccccccccc}
\hline
$n$
& $5$ & $17$ & $19$ & $41$ & $53$ & $63$ & $75$ & $89$ & $125$\\
\hline
\text{Paley order}
& $9$ & $17$ & $37$ & $81$ & $53$ & $125$ & $149$ & $89$ & $125$\\
\hline
\text{used as order}
& $2n-1$ & $n$ & $2n-1$ & $2n-1$ & $n$
& $2n-1$ & $2n-1$ & $n$ & $n$\\
\hline
$R(B_n,B_n)$
& $21$ & $69$ & $77$ & $165$ & $213$ & $253$ & $301$ & $357$ & $501$\\
\hline
\end{tabular}
\end{center}

Condition~(ii) supplies additional examples via Mathon's construction.
For instance, take $q=3$ and $p=q+2=5$. The Paley graph of order $5$
and Lemma~\ref{lem:mathon-one-step} give a PC-graph of order $45$.
Taking $2n-1=45$ yields $n=23$, and $4n+1=93=3\cdot31$ is not a
sum of two squares. Hence
\[
R(B_{23},B_{23})=93.
\]

The following table lists further applications of condition~(ii).
In every column, $q$ is a prime power congruent to $3\pmod4$ and
$p=q+2$ is a prime power congruent to $1\pmod4$, so the required
PC-graph of order $p$ is a Paley graph. The last row gives a prime
$\pi\equiv3\pmod4$ that divides $4n+1$ exactly once, verifying the
sum-of-two-squares obstruction in each case.

\begin{center}
\small
\setlength{\tabcolsep}{5pt}
\renewcommand{\arraystretch}{1.15}
\begin{tabular}{c|cccccc}
\hline
$n$ & $23$ & $221$ & $787$ & $1573$ & $6613$ & $10571$\\
\hline
$(p,q)$ & $(5,3)$ & $(9,7)$ & $(13,11)$ & $(13,11)$
& $(25,23)$ & $(29,27)$\\
\hline
\text{PC order } $pq^2$ & $45$ & $441$ & $1573$ & $1573$
& $13225$ & $21141$\\
\hline
\text{used as order} & $2n-1$ & $2n-1$ & $2n-1$ & $n$
& $2n-1$ & $2n-1$\\
\hline
$R(B_n,B_n)$ & $93$ & $885$ & $3149$ & $6293$ & $26453$ & $42285$\\
\hline
$\pi$ & $3$ & $3$ & $47$ & $7$ & $7$ & $3$\\
\hline
\end{tabular}
\end{center}

The infinite family in \eqref{eq:explicit-diagonal-family} follows
from the Paley inputs in condition~(i); its proof does not require
Mathon's construction.

\section{Nearly diagonal upper bound and Hadamard construction}\label{sec:proof-near-exact}

This section proves the nearly diagonal results. Subsection~\ref{sec:proof-upper} establishes the upper bound stated in Theorem~\ref{thm:near-exact}, and Subsection~\ref{sec:proof-hadamard} presents the Hadamard lower bound construction.

\subsection{Upper bound for \texorpdfstring{$R(B_{n-2},B_n)$}{R(B(n-2),B(n))}}\label{sec:proof-upper}

The original proof of Rousseau and Sheehan \cite{RS78} for the case
$n\equiv 2\pmod 3$ was based on a counting argument using Goodman's
identity. Here, we further analyze the structure of a hypothetical
counterexample and show that it would necessarily be a strongly
regular graph with parameters $\srg(4n-3,2n-2,n-3,n)$.
Using eigenvalue methods, we prove that such a graph can exist
only when $n=6$, thereby removing the congruence condition and
establishing the upper bound for all $n\ge 3$ with $n\ne 6$. 

\medskip
\begin{pf}[Proof of Theorem~\ref{thm:near-exact}]
Suppose for a contradiction that there exists a graph $G$ on $p=4n-3$ vertices such that $G$ contains no $B_{n-2}$ and $\overline G$ contains no $B_n$. Let $e=e(G)$ and $f=\binom p2.$

Let $M$ denote the total number of monochromatic triangles in the red-blue coloring of $K_p$ induced by $G$ and $\overline G$. Since $G$ contains no $B_{n-2}$, every edge of $G$ lies in at most $n-3$ triangles of $G$. Similarly, since $\overline G$ contains no $B_n$, every edge of $\overline G$ lies in at most $n-1$ triangles of $\overline G$. Therefore
\begin{equation}\label{eq:upperM}
M\le \frac{(n-3)e+(n-1)(f-e)}{3}.
\end{equation}
It is convenient to set $e=\frac f2+ x.$
Then \eqref{eq:upperM} becomes
\begin{equation}\label{eq:upperM2}
M\le \frac{(n-2)f-2x}{3}.
\end{equation}

On the other hand, Goodman's identity \cite{Goodman1959} gives
\begin{equation}\label{eq:goodman}
M=\binom p3-\frac12\sum_{i=1}^p d_i(p-1-d_i),
\end{equation}
where $d_i=d_G(v_i)$ for $1\le i\le p$. Let $d_i=\frac{p-1}{2}+\varepsilon_i$.
Since $f = \frac{p(p-1)}{2}$, we have
\[
f+\sum_{i=1}^p\varepsilon_i=\sum_{i=1}^p d_i=2e=f+2x.
\]
Therefore,
\begin{equation}\label{eq:sumeps}
\sum_{i=1}^p \varepsilon_i = 2x.
\end{equation}
Furthermore, $d_i(p-1-d_i)=(\frac{p-1}{2}+\varepsilon_i)(\frac{p-1}{2}-\varepsilon_i)=\frac{(p-1)^2}{4}-\varepsilon_i^2.$ Substituting this into \eqref{eq:goodman} and using $p=4n-3$ gives
\begin{equation}\label{eq:lowerM}
\begin{aligned}
M
&= \frac{p(p-1)(p-2)}{6}
   -\frac{p(p-1)^2}{8}
   +\frac{1}{2}\sum_{i=1}^p \varepsilon_i^2 \\
&= \frac{p(p-1)(p-5)}{24}
   +\frac{1}{2}\sum_{i=1}^p \varepsilon_i^2  = \frac{(n-2)f}{3}
   +\frac{1}{2}\sum_{i=1}^p \varepsilon_i^2 .
\end{aligned}
\end{equation}

Comparing \eqref{eq:upperM2} and \eqref{eq:lowerM}, we obtain
\begin{equation}\label{eq:keyineq}
3\sum_{i=1}^p\varepsilon_i^2+4x \le 0.
\end{equation}

Since the $\varepsilon_i$ are integers, \eqref{eq:sumeps} implies
$
\sum_i\varepsilon_i^2\ge \sum_i|\varepsilon_i|\ge \left|\sum_i\varepsilon_i\right|=2|x|.
$
Then the left-hand side of \eqref{eq:keyineq} is no less than $6|x|+4x \ge 2|x|$. Hence we have
\[
x=0
\quad\text{and}\quad
\sum_i\varepsilon_i^2=0.
\]
Thus $d_i=2n-2$ for all $i$, implying that $G$ is $(2n-2)$-regular. Moreover, equality must hold in \eqref{eq:upperM}. Therefore every edge of $G$ lies in exactly $n-3$ triangles of $G$, and every edge of $\overline G$ lies in exactly $n-1$ triangles of $\overline G$.

Now suppose $xy\in E(\overline G)$. Then the number of common neighbors of $x$ and $y$ in $\overline G$ satisfies
\[
p-2-d_G(x)-d_G(y)+|N_G(x)\cap N_G(y)|=n-1.
\]
Using $p=4n-3$ and $d_G(x)=d_G(y)=2n-2$, we have $|N_G(x)\cap N_G(y)|=n.$ Thus $G$ must be a strongly regular graph
\begin{equation}\label{eq:badparams}
\srg(4n-3,2n-2,n-3,n).
\end{equation}
Since its parameter $\mu=n$ is positive, $G$ is connected, so its principal eigenvalue $2n-2$ has multiplicity one.

We will show that no such strongly regular graph exists unless $n=6$. Let $A$ be the adjacency matrix of $G$. By Lemma~\ref{lem:srg-spectral},
every eigenvalue $\rho\ne 2n-2$ of $G$ satisfies
\begin{equation}\label{eq:bad-eigen-equation}
\rho^2+3\rho-(n-2)=0.
\end{equation}
Thus the two possible nontrivial eigenvalues are
\[
\theta=\frac{-3+\sqrt{4n+1}}{2},
\qquad
\tau=\frac{-3-\sqrt{4n+1}}{2}.
\]
We first show that $4n+1$ must be a square. Otherwise, $\theta$ and $\tau$ would be conjugate quadratic irrationalities. Since the characteristic polynomial of the integer matrix $A$ has integer coefficients, these two conjugate roots would occur with the same multiplicity. As the total multiplicity of the nontrivial eigenvalues is $p-1=4n-4,$ each of $\theta$ and $\tau$ would have multiplicity $2n-2$. Taking the trace of $A$ would then give
\[
0=(2n-2)+(2n-2)(\theta+\tau)=(2n-2)-3(2n-2)=-4(n-1),
\]
which is impossible for $n\ge 3$. Hence $4n+1$ is a square. Write $4n+1=t^2.$ Then
\[
\theta=\frac{-3+t}{2},
\quad
\tau=\frac{-3-t}{2}.
\]
Let $m_\theta$ and $m_\tau$ be the multiplicities of $\theta$ and $\tau$, respectively. Since $m_\theta+m_\tau=p-1=4n-4$ and $\operatorname{tr}(A)=0$, we have $(2n-2)+m_\theta\theta+m_\tau\tau=0.$ Solving these two linear equations gives
\[
m_\theta=2(n-1)+\frac{4(n-1)}{t},
\quad
m_\tau=2(n-1)-\frac{4(n-1)}{t}.
\]
Since $m_\theta$ and $m_\tau$ are integers, $t$ divides $4(n-1)$. But $4(n-1)=4n+1-5=t^2-5.$ Hence $t\mid 5$. Since $n\ge 3$, we have $t>3$, and therefore $t=5$. This forces \(4n+1=25\), so $n=6$.

Consequently, for every $n\ge 3$ with $n\ne 6$, no such counterexample exists. Therefore the upper bound
$
R(B_{n-2},B_n)\le 4n-3
$ follows as desired.
This completes the proof.
\end{pf}

\begin{remark}\label{rem:n6}

The exclusion of $n=6$ is necessary. The triangular graph
$T(7)=L(K_7)$ has parameters $\srg(21,10,5,4)$
\cite{Hoffman1960}, so its complement has parameters
$\srg(21,10,3,6)$. Thus $G=\overline{T(7)}$ contains no $B_4$,
and $\overline G$ contains no $B_6$. Consequently,
\[
R(B_4,B_6)\ge22>21=4\cdot6-3,
\]
showing that the upper bound fails in this case.

\end{remark}
\subsection{The Hadamard lower bound construction}\label{sec:proof-hadamard}

We begin by recalling the matrix definition needed for this construction.

\begin{definition}\label{def:hadamard}
A \emph{Hadamard matrix} of order $N$ is a matrix
$H=(h_{ij})_{1\le i,j\le N}$ with entries in $\{\pm1\}$ such that
\[
HH^{\mathsf T}=NI.
\]
It is symmetric if $H=H^{\mathsf T}$.
\end{definition}

We use symmetric Hadamard matrices whose diagonal entries are all equal
to $1$. The construction follows Case~1 of the proof of
Theorem~\ref{thm:Bn-PC}, with $M=C+I$ replaced by such a matrix.
Row orthogonality then gives the common-neighbor counts required for
the nearly diagonal lower bound.

\medskip
\begin{pf}[Proof of Lemma~\ref{cor:near-hadamard-exact}]
Let $n\ge 3$, and suppose that there exists a symmetric Hadamard matrix $H=(h_{ij})_{1\le i,j\le N}$ of order $N=2n-2$ with $h_{ii}=1$ for every $i$. We construct a graph $G$ on $V(G)=[N]\times\{+1,-1\}.$ Thus $|V(G)|=2N=4n-4$. For two distinct vertices $(i,\varepsilon)$ and
$(j,\delta)$, where $\varepsilon,\delta\in\{+1,-1\}$, define $(i,\varepsilon)(j,\delta)\in E(G)$ if and only if $i\ne j$ and $\varepsilon\delta=h_{ij}.$

The same common-neighbor counting argument as in Case~1 of the proof of Theorem~\ref{thm:Bn-PC} gives the following expressions. If $i\ne j$ and $(i,\varepsilon)(j,\delta)\in E(G)$, then the number of common neighbors of this edge in $G$ is
\[
\alpha_{ij}=\bigl|\{k\in [N]\setminus\{i,j\}:h_{ik}h_{jk}=h_{ij}\}\bigr|.
\]
Also, if $i\ne j$ and $(i,\varepsilon)(j,\delta)\in E(\overline G)$, then the number of common neighbors of this edge in $\overline G$ is
\[
\beta_{ij}=\bigl|\{k\in [N]\setminus\{i,j\}:h_{ik}h_{jk}=-h_{ij}\}\bigr|.
\]
For fixed $i\neq j$, the indices $k\neq i,j$ are partitioned into those with $h_{ik}h_{jk}=h_{ij}$ (counted by $\alpha_{ij}$) and those with $h_{ik}h_{jk}=-h_{ij}$ (counted by $\beta_{ij}$). Hence 
\begin{equation}\label{eq:partition}
\alpha_{ij}+\beta_{ij}=N-2.
\end{equation}

As in Case~1 of the proof of Theorem~\ref{thm:Bn-PC}, the endpoints of every edge of the form $(i,+1)(i,-1)$ in $\overline G$ have no common neighbor in $\overline G$; thus we only need to compute $\alpha_{ij}$ and $\beta_{ij}$.

Since $HH^{\mathsf T}=NI$, the $i$-th and $j$-th rows are orthogonal for $i\ne j$, so
\[
\sum_{k=1}^N h_{ik}h_{jk}=0.
\]
Using symmetry and $h_{ii}=h_{jj}=1$, the two terms with $k=i$ and $k=j$ are both equal to $h_{ij}$. Hence
\begin{equation}\label{eq:hadamard-excluded-row-sum}
\sum_{k\ne i,j}h_{ik}h_{jk}=-2h_{ij}.
\end{equation}

Among the $N-2$ indices $k\ne i,j$, exactly $\alpha_{ij}$ of them satisfy $h_{ik}h_{jk}=h_{ij}$, and the remaining $\beta_{ij}=N-2-\alpha_{ij}$ satisfy $h_{ik}h_{jk}=-h_{ij}$. Dividing \eqref{eq:hadamard-excluded-row-sum} by $h_{ij}$ gives 
\[
\alpha_{ij}-(N-2-\alpha_{ij})=-2.
\]
This yields
\[
\alpha_{ij}=\frac{N-4}{2}=\frac{2n-2-4}{2}=n-3,\quad\beta_{ij}=N-2-\alpha_{ij}=n-1.
\]
Thus, from the definitions of $\alpha_{ij}$ and $\beta_{ij}$, every edge of $G$ has exactly $n-3$ common neighbors in $G$, implying $B_{n-2}\nsubseteq G$; and in $\overline G$, every edge with $i\ne j$ has exactly $n-1$ common neighbors, whereas edges of the form $(i,+1)(i,-1)$ have none, hence $B_{n}\nsubseteq\overline G$.

Thus we have constructed a graph on $4n-4$ vertices such that $G$ contains no $B_{n-2}$ and $\overline G$ contains no $B_n$. Consequently,
$
R(B_{n-2},B_n)\ge 4n-3.
$
\end{pf}

To apply this lower bound, we construct an infinite family of symmetric
Hadamard matrices with diagonal entries all equal to $1$.

\begin{lemma}\label{lem:binary-hadamard}
For every integer $\ell\ge 1$, there exists a symmetric Hadamard matrix $H$ of order $2^{2\ell}$ whose diagonal entries are all equal to $1$.
\end{lemma}

\begin{pf}
Let
\[
H_4=
\begin{pmatrix}
1&1&1&-1\\
1&1&-1&1\\
1&-1&1&1\\
-1&1&1&1
\end{pmatrix}.
\]
A direct computation gives $H_4=H_4^{\mathsf T}$ and
$H_4H_4^{\mathsf T}=4I_4$, and all diagonal entries of $H_4$ are $1$.
For $\ell\ge1$, take $H=H_4^{\otimes\ell}$. The standard identities
for Kronecker products show that $H$ is symmetric, has all diagonal
entries equal to $1$, and satisfies
\[
HH^{\mathsf T}=(H_4H_4^{\mathsf T})^{\otimes\ell}
=4^\ell I_{4^\ell}.
\]
Thus $H$ is the required matrix of order $4^\ell=2^{2\ell}$.
\end{pf}

\begin{pf}[Proof of Corollary~\ref{cor:hadamard-exact}]
Since $n=2^{2\ell-1}+1 \neq 6$, Theorem~\ref{thm:near-exact} gives the upper bound $R(B_{n-2},B_n) \le 4n-3$.
By Lemma~\ref{lem:binary-hadamard}, there exists a symmetric Hadamard matrix of order $2^{2\ell}$ with all diagonal entries equal to $1$. Setting $2n-2 = 2^{2\ell}$ yields $n = 2^{2\ell-1}+1$, so Lemma~\ref{cor:near-hadamard-exact} provides the matching lower bound $R(B_{n-2},B_n) \ge 4n-3$.
Therefore $R(B_{n-2},B_n) = 4n-3$. 
This completes the proof.
\end{pf}

\section{Concluding remarks}
\label{sec:concluding}
We have introduced PC-graphs as a new algebraic source for diagonal book Ramsey constructions, yielding a new infinite family of exact values \(R(B_n,B_n)=4n+1\). We have also removed the congruence condition in the nearly diagonal upper bound \(R(B_{n-2},B_n)\le 4n-3\) and matched it by a symmetric Hadamard construction whenever a suitable matrix of order \(2n-2\) exists. Can one prove that $R(B_{n-2},B_n)=4n-3$ for all $n\ge3$ with $n\ne6$? A related question is to characterize the integers $n\ge3$ for which there exists a symmetric Hadamard matrix of
order $2n-2$ with all diagonal entries equal to $1$. More generally, which other classes of strongly regular graphs or conference matrices can yield further infinite families of exact diagonal book Ramsey numbers?

\section*{Declaration on the use of AI}
The authors independently developed the mathematical results and constructions presented in this paper, except for the construction in Case 2 of the proof of Theorem~\ref{thm:Bn-PC}. For this case, ChatGPT
(OpenAI) proposed the key four-block adjacency-matrix construction based on a pseudo-cyclic strongly regular graph. The authors subsequently verified and refined the construction and its
accompanying mathematical arguments. The authors take full responsibility for the correctness and final content of the manuscript.

\end{spacing}

\end{document}